\documentclass[10pt,twoside,a4paper]{article}
\usepackage[italian,english]{babel}
\usepackage{amsmath,amsfonts,amssymb,amscd,lscape,amstext,amsthm}
\usepackage{tabularx}
\newtheorem{lemma}{\rm \indent LEMMA}[section]
\newtheorem{definition}{\rm \indent D\small EFINITION}[section]
\newtheorem{theorem}[lemma]{\rm \indent T\small HEOREM}
\newtheorem{remark}[lemma]{\rm \indent Remark}
\newtheorem{proposition}[lemma]{\rm \indent Proposition}
\numberwithin{equation}{section}
\begin{document}
\title{\bf{Lipschitz stability for the inverse conductivity problem for a conformal class of anisotropic conductivities }}
\author{Romina Gaburro\thanks{Department of Mathematics and Statistics, University of Limerick, Ireland.  Email: romina.gaburro@ul.ie}\qquad
Eva Sincich\thanks{Laboratory for Multiphase Processes, University
of Nova Gorica, Slovenia. Email: eva.sincich@ung.si}}

\date{}
\maketitle \footnotesize{\textbf{Abstract.} We consider the stability issue of the inverse conductivity
problem for a conformal class of anisotropic conductivities in
terms of the local Dirichlet-to-Neumann map. We extend here the
stability result obtained by Alessandrini and Vessella in \textnormal{Advances
in Applied Mathematics 35:207-241}, where the authors considered the piecewise
constant isotropic case.

\section{\normalsize{Introduction}}\label{sec1}
\setcounter{equation}{0} 

\normalsize In the present paper we study the
stability issue for the inverse conductivity problem in the
presence of anisotropic conductivity which is \textit{a-priori}
known to depend linearly on a unknown piecewise-constant function.
Let us start by recalling the basic formulation of the inverse
conductivity problem.

 In absence of internal sources, the electrostatic
potential $u$ in a conducting body, described by a domain
$\Omega\subset{\mathbb R}^n$, is governed by the elliptic equation
\begin{equation}\label{eq conduttivita'}
\mbox{div}(\sigma\nabla{u})=0\qquad\mbox{in}\quad\Omega,
\end{equation}
where the symmetric, positive definite matrix $\sigma=\sigma(x)$,
$x\in\Omega$ represents the (possibly anisotropic) electric
conductivity. The inverse conductivity problem consists of finding
$\sigma$ when the so called Dirichlet-to-Neumann (D-N) map
$$
\Lambda_{\sigma}:u\vert_{\partial\Omega}\in{H}^{\frac{1}{2}}(\partial\Omega)
\longrightarrow{\sigma}\nabla{u}\cdot\nu\vert_{\partial\Omega}\in{H}^{-\frac{1}{2}}(\partial\Omega)
$$
is given for any $u\in{H}^{1}(\Omega)$ solution to (1.1). Here,
 $\nu$ denotes the unit outer normal to $\partial\Omega$. If
measurements can be taken only on one portion $\Sigma$ of
$\partial\Omega$, then the relevant map is called the local
Dirichlet-to-Neumann map. Let $\Sigma$ be a non-empty open portion
of $\partial\Omega$ and let us introduce the subspace of
$H^{\frac{1}{2}}(\partial\Omega)$
\begin{equation}\label{Hco}
H^{\frac{1}{2}}_{co}(\Sigma)=\big\{f\in
H^{\frac{1}{2}}(\partial\Omega) \:\vert\:\textnormal{supp}
\:f\subset\Sigma\big\}.
\end{equation}
The local D-N map is given, in its weak
formulation, as the operator $\Lambda_{\sigma}^{\Sigma}$ such that
\begin{equation}
\langle\Lambda_{\sigma}^{\Sigma}\;u,\;\phi\rangle=\int_{\Omega}\sigma\;\nabla
u\cdot\nabla\phi,
\end{equation}
for any $u,\;\phi\in{H}^{1}(\Omega)$,
$u\vert_{\partial\Omega},\;\phi\vert_{\partial\Omega}\in{H}_{co}^{\frac{1}{2}}(\Sigma)$
and $u$ is a weak solution to \eqref{eq conduttivita'}.

The problem of recovering the conductivity of a body by taking
measurements of voltage and current on its surface has came to be
known as Electrical Impedance Tomography (EIT). Different
materials display different electrical properties, so that a map
of the conductivity $\sigma(x)$, $x\in\Omega$ can be used to investigate internal properties of
$\Omega$. EIT has many important applications in fields such as
geophysics, medicine and non--destructive testing of materials.
The first mathematical formulation of the inverse conductivity
problem is due to A. P. Calder\'{o}n \cite{C}, 
where he addressed the problem of whether it is possible to
determine the (isotropic) conductivity $\sigma = \gamma I$ by the D-N map. Although
Calder\'{o}n studied the problem of determining $\sigma$ from the
knowledge of the quadratic form

\[Q_{\gamma}(u)=\int_{\Omega} \gamma |\nabla u|^2,\]

where $u$ is a solution to \eqref{eq conduttivita'}, it is well
known that the knowledge of $Q_{\sigma}$ is equivalent to the
knowledge of $\Lambda_{\sigma}$ by

\[Q_{\gamma}(u)=\left<\Lambda_{\sigma}u,u\right>,\qquad\textnormal{for\:every}\:u\in H^{1}(\Omega),\]

where $\sigma = \gamma I$. Here $\left<\cdot,\cdot\right>$ denotes the dual
pairing between $H^{1/2}(\partial\Omega)$ and its dual
$H^{-1/2}(\partial\Omega)$, with respect to the $L^{2}$ scalar
product. \cite{C} opened the way to the solution to the uniqueness
issue where one is asking whether $\sigma$ can be determined by
the knowledge of $\Lambda_{\sigma}$ (or
$\Lambda_{\sigma}^{\Sigma}$ in the case of local measurements). As main contributions in this respect we mention the papers by
Kohn and Vogelius \cite{K-V1,K-V2}, Sylvester and Uhlmann \cite{S-U} and Nachman \cite{Na}.  We
refer to \cite{B}, \cite{C-I-N} and \cite{U} for an overview of
recent developments regarding the issues of uniqueness and
reconstruction of the conductivity.\\
Regarding the stability, Alessandrini proved in \cite{A} that,
assuming $n\geq 3$ and \textit{a-priori} bounds on $\gamma$ of the
form

\begin{equation}\label{regularity bounds on sigma}
||\gamma||_{H^s(\Omega)}\leq E,
\quad\textnormal{for\:some}\:s>\frac{n}{2}+2,
\end{equation}

$\gamma$ depends continuously on $\Lambda_{\sigma}$ with a modulus
of continuity of logarithmic type. In \cite{A1}, \cite{A2} the
same author subsequently proved that a similar stability estimate
holds when the \textit{a-priori} bound \eqref{regularity bounds on
sigma} is replaced by

\begin{equation}\label{lower regularity bounds gamma}
||\gamma||_{W^{2,\infty}(\Omega)}\leq E.
\end{equation}

For the two-dimensional case, logarithmic type stability estimates
were obtained in \cite{B-B-R}, \cite{B-F-R} and \cite{Liu}. Unfortunatelly, all the above
results have the common inconvenient logarithmic type of stability
which cannot be avoided \cite{A3}. In fact Mandache \cite{Ma} showed that the
logarithmic stability is the best possible, in any dimension
$n\geq 2$ if \textit{a-priori} assumptions of the form

\begin{equation}\label{regularity bounds on sigma strong}
||\gamma||_{C^k(\Omega)}\leq E
\end{equation}

for any $k=0,1,2,\dots$ are assumed. It seems therefore reasonable
to think that, in order to restore stability in a really
(Lipschitz) stable fashion, one needs to replace in some way the
\textit{a-priori} assumptions expressed in terms of regularity
bounds such as \eqref{regularity bounds on sigma strong}, with
\textit{a-priori} pieces of information of a different type.
Alessandrini and Vessella showed in \cite{A-V} that $\gamma$
depends in a Lipschitz continuous fashion upon the local D-N map,
by assuming that $\gamma$ is a function \textit{a-priori} known to
be piecewise constant

\begin{equation}\label{piecewise constant sigma}
\gamma(x)=\sum_{j=1}^{N}\gamma_j \chi_{D_j}(x),
\end{equation}

where each subdomain of $\Omega$, $D_j$, $j=1,\dots , N$ is
\textit{given} and each number $\gamma_j$, $j=1,\dots , N$ is
\textit{unknown}. From a medical imaging point of view, each $D_j$
may represent the area occupied by different tissues or organs
and one can think that the geometrical configuration of each $D_j$
is given by means of other imaging techniques such as MRI for
example. Since most tissues in the human body are anisotropic, the
present authors, motivated by the work in \cite{A-V} and its
medical application, consider here the more general case of an
\textit{anisotropic} conductivity of type

\[\sigma(x) = \gamma(x) A(x),\]

where $A(x)$ is a known, matrix valued function which is Lipschitz continuous and
$\gamma(x)$ is of type \eqref{piecewise constant sigma}. The authors would like to stress out that anisotropic conductivity appears in nature, for example as a homogenization limit in layered or fibrous structures
such as rock stratum or muscle, as a result of crystalline structure or of deformation of an
isotropic material, therefore the case treated in this paper seems to be a natural extension of \cite{A-V} relevant to several applications. For related results in the anisotropic case we also refer to \cite{A-G}, \cite{A-G1}, \cite{A-L-P}, \cite{Be}, \cite{F-K-R}, \cite{G-L}, \cite{L} and \cite{La-U}).  The present paper improves
upon the results obtained in \cite{A-V} in the sense that the
global Lipschitz stability estimate obtained there is here adapted
to a special anisotropic type of conductivity. The precise
assumptions shall be illustrated in section \ref{sec2}. We also recall \cite{B-F}, \cite{B-F-V} and \cite{B-dH-Q} where similar Lipschitz stability results have been obtained for complex conductivity, the Lam\'{e} parameters and for a Schr\"{o}dinger type of equation respectively.
\\
For a more in-dept description and consideration of the stability
issue and related open problems in the inverse conductivity
problem we refer to \cite{A3} and \cite{A-V}.

Our approach follows the one by Alessandrini and Vessella \cite{A-V} of constructing singular solutions and studying their asymptotic behaviour when the singularity approaches the discontinuity interfaces. However, in order to deal with the present structure of conductivity we had to develop original asymptotic analysis estimates and an accurate quantitative control of the error terms which represent a novel feature in the treatment of anisotropic type of conductivity.

The paper is organized as follows. Our main assumptions and our main result (Theorem \ref{teorema principale}) are contained in section 2, where the proof of Theorem \ref{teorema principale} is contained in section 3. This section also lists the two main results (Theorem \ref{teorema stime asintotiche} and Proposition \ref{proposizione unique continuation finale}) needed to build the machinery for the proof of Theorem \ref{teorema principale}.  Theorem \ref{teorema stime asintotiche} provides original asymptotic estimates for the Green function of the conductivity equation, for conductivities belonging to a special anisotropic conformal class $\mathcal{C}$, at the interfaces between the given domains $D_j$, where the conductivity is discontinuous. Proposition \ref{proposizione unique continuation finale} provides estimates of unique continuation of the solution to the conductivity equation for conductivities in $\mathcal{C}$. Section 4 is devoted to the proof of Theorem \ref{teorema stime asintotiche} and Proposition \ref{proposizione unique continuation finale}. For the proof of Theorem \ref{teorema stime asintotiche} we provide the explicit form of the fundamental solution for the conformal anisotropic two-phase case with flat interface. The proof of Proposition \ref{proposizione unique continuation finale} is a straight forward consequence Proposition \ref{proposizione pre unique continuation} which we state in this section. The proof of the latter is independent from the presence of anisotropy in the conductivity, therefore we refer to \cite{A-V} for a full proof of it. In this paper we point out the main facts on which the proof is based on only.

\section{\normalsize{Main Result}}\label{sec2}
\setcounter{equation}{0}
\subsection{\normalsize{Notation and definitions}}

In several places within this manuscript it will be useful to single out one coordinate
direction. To this purpose, the following notations for
points $x\in \mathbb{R}^n$ will be adopted. For $n\geq 3$,
a point $x\in \mathbb{R}^n$ will be denoted by
$x=(x',x_n)$, where $x'\in\mathbb{R}^{n-1}$ and $x_n\in\mathbb{R}$.
Moreover, given a point $x\in \mathbb{R}^n$,
we shall denote with $B_r(x), B_r'(x)$ the open balls in
$\mathbb{R}^{n},\mathbb{R}^{n-1}$ respectively centred at $x$ with radius $r$
and by $Q_r(x)$ the cylinder

\[Q_r(x)=B_r'(x')\times(x_n-r,x_n+r).\]

We shall also denote

\begin{eqnarray*}
& & \mathbb{R}^n_+ = \{(x',x_n)\in \mathbb{R}^n| x_n>0 \};\quad\mathbb{R}^n_- = \{(x',x_n)\in \mathbb{R}^n| x_n<0 \};\\
& & B^+_r = B_r\cap\mathbb{R}^n_+;\quad B^-_r = B_r\cap\mathbb{R}^n_-;\\
& & Q^+_r = Q_r\cap\mathbb{R}^n_+;\quad Q^{-}_r = Q_r\cap\mathbb{R}^n_-.
\end{eqnarray*}

In the sequel, we shall make a repeated use of quantitative
notions of smoothness for the boundaries of various domains. Let
us introduce the following notation and definitions.

\begin{definition}\label{def Lipschitz boundary}
Let $\Omega$ be a domain in $\mathbb R^n$. We say that a portion
$\Sigma$ of $\partial\Omega$ is of Lipschitz class with constants
$r_0,L$ if for any $P\in\partial\Sigma$ there exists a rigid
transformation of $\mathbb R^n$ under which we have $P\equiv0$ and
$$\Omega\cap Q_{r_0}=\{x\in Q_{r_0}\,:\,x_n>\varphi(x')\},$$
where $\varphi$ is a Lipschitz function on $B'_{r_0}$ satisfying

\[\varphi(0)=|\nabla_{x'}\varphi(0)|=0;\qquad
\|\varphi\|_{C^{0,1}(B'_{r_0})}\leq Lr_0.\]

It is understood that $\partial\Omega$ is of Lipschitz class with
constants $r_0,L$ as a special case of $\Sigma$, with
$\Sigma=\partial\Omega$.
\end{definition}

\begin{definition}\label{def C1 alpha boundary}
Let $\Omega$ be a domain in $\mathbb R^n$. Given $\alpha$,
$\alpha\in(0,1]$, we say that a portion $\Sigma$ of
$\partial\Omega$ is of class $C^{1,\alpha}$ with constants $r_0,M$
if for any $P\in\Sigma$ there exists a rigid transformation of
$\mathbb R^n$ under which we have $P=0$ and
$$\Omega\cap Q_{r_0}=\{x\in Q_{r_0}\,:\,x_n>\varphi(x')\},$$
where $\varphi$ is a $C^{1,\alpha}$ function on $B'_{r_0}$
satisfying
\[\varphi(0)=|\nabla_{x'}\varphi(0)|=0;\qquad\|\varphi\|_{C^{1,\alpha}(B'_{r_0})}\leq Mr_0,\]

where we denote
\begin{eqnarray}\label{chap2:8}
\|\varphi\|_{C^{1, \alpha}(
B^{'}_{r_0})}&=&\|\varphi\|_{L^{\infty}( B^{'}_{r_0})}+
{r_0}\|\nabla \varphi\|_{L^{\infty}( B^{'}_{r_0})}
+{r_0}^{1+\alpha}\sup_{\substack {x,y  \in B^{'}_{r_0}\\x\ne y
}}\frac{|\nabla \varphi (x)-\nabla \varphi (y)|}{|x-y|^{\alpha}}\
.\nonumber
\end{eqnarray}

\end{definition}


Let us rigorously define the local D-N map.\\

\begin{definition}
Let $\Omega$ be a domain in $\mathbb{R}^n$ with Lipschitz boundary
$\partial\Omega$ and $\Sigma$ an open non-empty subset of
$\partial\Omega$. Assume that $\sigma\in
L^{\infty}(\Omega\:,Sym_{n})$ satisfies the ellipticity condition
\begin{eqnarray}\label{ellitticita'sigma}
\lambda^{-1}\vert\xi\vert^{2}\leq{\sigma}(x)\xi\cdot\xi\leq\lambda\vert\xi\vert^{2},
& &for\:almost\:every\:x\in\Omega,\nonumber\\
& &for\:every\:\xi\in\mathbb{R}^{n}.
\end{eqnarray}

The local Dirichlet-to-Neumann map associated to $\sigma$ and
$\Sigma$ is the operator
\begin{equation}\label{mappaDN}
\Lambda_{\sigma}^{\Sigma}:H^{\frac{1}{2}}_{co}(\Sigma)\longrightarrow
{H}^{-\frac{1}{2}}_{co}(\Sigma)
\end{equation}
 defined by
\begin{equation}\label{def DN locale}
<\Lambda_{\sigma}^{\Sigma}\:g,\:\eta>\:=\:\int_{\:\Omega}
\sigma(x) \nabla{u}(x)\cdot\nabla\phi(x)\:dx,
\end{equation}
for any $g$, $\eta\in H^{\frac{1}{2}}_{co}(\Sigma)$, where
$u\in{H}^{1}(\Omega)$ is the weak solution to
\begin{displaymath}
\left\{ \begin{array}{ll} \textnormal{div}(\sigma(x)\nabla
u(x))=0, &
\textrm{$\textnormal{in}\quad\Omega$},\\
u=g, & \textrm{$\textnormal{on}\quad{\partial\Omega},$}
\end{array} \right.
\end{displaymath}
and $\phi\in H^{1}(\Omega)$ is any function such that
$\phi\vert_{\partial\Omega}=\eta$ in the trace sense. Here we
denote by $<\cdot,\:\cdot>$ the $L^{2}(\partial\Omega)$-pairing
between $H^{\frac{1}{2}}_{co}(\Sigma)$ and its dual
$H^{-\frac{1}{2}}_{co}(\Sigma)$.
\end{definition}
Note that, by \eqref{def DN locale}, it is easily verified that
$\Lambda^{\Sigma}_{\sigma}$ is selfadjoint. We shall denote by
$\parallel\cdot\parallel_{*}$ the norm on the Banach space of
bounded linear operators between $H^{\frac{1}{2}}_{co}(\Sigma)$
and $H^{-\frac{1}{2}}_{co}(\Sigma)$.

\subsection{\normalsize{Our assumptions}}\label{subsection assumptions}
We give here the precise assumptions for the domain $\Omega$ under investigation and its conductivity $\sigma$. The dimension of the space for $\Omega$ is denoted by $n$ and
for sake of simplicity is only consider $n\geq 3$.

\subsubsection{\normalsize{Assumptions about the domain $\Omega$}}\label{subsec assumption domain}

\begin{enumerate}

\item We assume that $\Omega$ is a domain in $\mathbb{R}^n$
satisfying

\begin{equation}\label{assumption Omega}
|\Omega|\leq N r_0 ^n,
\end{equation}

where $|\Omega|$ denotes the Lebesgue measure of $\Omega$.

\item We assume that $\partial\Omega$ is of Lipschitz class with
constants $r_0$, $L$.

\item We fix an open non-empty subset $\Sigma$ of $\partial\Omega$
(where the measurements in terms of the local D-N map are taken).

\item \[\bar\Omega = \bigcup_{j=1}^{N}\bar{D}_j,\]

where $D_j$, $j=1,\dots , N$ are known open sets of
$\mathbb{R}^n$, satisfying the conditions below.

\begin{enumerate}
\item $D_j$, $j=1,\dots , N$ are connected and pairwise
nonoverlapping.

\item $\partial{D}_j$, $j=1,\dots , N$ are of Lipschitz class with
constants $r_0$, $L$.

\item There exists one region, say $D_1$, such that
$\partial{D}_1\cap\Sigma$ contains a $C^{1,\alpha}$ portion
$\Sigma_1$ with constants $r_0$, $M$.

\item For every $i\in\{2,\dots , N\}$ there exists $j_1,\dots ,
j_K\in\{1,\dots , N\}$ such that

\begin{equation}\label{catena dominii}
D_{j_1}=D_1,\qquad D_{j_K}=D_i.
\end{equation}

In addition we assume that, for every $k=1,\dots , K$,
$\partial{D}_{j_k}\cap \partial{D}_{j_{k-1}}$ contains a
$C^{1,\alpha}$ portion $\Sigma_k$ (here we agree that
$D_{j_0}=\mathbb{R}^n\setminus\Omega$), such that

\[\Sigma_1\subset\Sigma,\]

\[\Sigma_k\subset\Omega,\quad\mbox{for\:every}\:k=2,\dots , K,\]

and, for every $k=1,\dots , K$, there exists $P_k\in\Sigma_k$ and
a rigid transformation of coordinates under which we have $P_k=0$
and

\begin{eqnarray}
\Sigma_k\cap{Q}_{r_{0}/3} &=&\{x\in
Q_{r_0/3}|x_n=\phi_k(x')\}\nonumber\\
D_{j_k}\cap {Q}_{r_{0}/3} &=&\{x\in
Q_{r_0/3}|x_n>\phi_k(x')\}\nonumber\\
D_{j_{k-1}}\cap {Q}_{r_{0}/3} &=&\{x\in
Q_{r_0/3}|x_n<\phi_k(x')\},
\end{eqnarray}

where $\phi_k$ is a $C^{1,\alpha}$ function on $B'_{r_o/3}$
satisfying

\[\phi_k(0)=|\nabla\phi_k(0)|=0\]

and

\[||\phi_k||_{C^{1,\alpha}(B'_{r_0})}\leq Mr_0.\]
\end{enumerate}
\end{enumerate}

\subsubsection{\normalsize{A-priori information on the conductivity $\gamma$: the class $\mathcal{C}$}}

\begin{definition}
We shall say that $\sigma\in\mathcal{C}$ if $\sigma$ is of type

\begin{equation}\label{a priori info su sigma}
\sigma_{A}(x)=\sum_{j=1}^{N}\gamma_{j}\:A(x)\chi_{D_j}(x),\qquad
x\in\Omega,
\end{equation}

where $\gamma_{j}$ are unknown real numbers, $D_j$, $j=1,\dots ,
N$ are the given subdomains introduced in section \ref{subsec assumption
domain} and

\begin{equation}\label{apriorigamma}
\bar{\gamma}\le \gamma_j\le  \bar{\gamma}^{-1} , \qquad \mbox{for
any}\:j=1,\dots n.
\end{equation}

$A(x)$ is a known Lipschitz matrix valued function satisfying

\begin{equation}\label{aprioria}
\|A\|_{C^{0,1}(\Omega)}\le \bar{A},
\end{equation}

where $\bar{A}>0$ is a constant and

\begin{eqnarray}\label{unifellip}
{\lambda}^{-1}|\xi|^2\le A(x)\xi\cdot\xi\le \lambda|\xi|^2 ,\qquad &&\mbox{for almost every}\ x\in \Omega, \\
 &&\mbox{for every}\ \xi\in\mathbb{R}^n \ . \nonumber
\end{eqnarray}

\end{definition}


\begin{definition}
Let $N$, $r_0$, $L$, $M$, $\alpha$, $\lambda, \bar{\gamma}$, $\bar{A}$ be given
positive numbers with $N\in\mathbb{N}$ and $\alpha\in(0,1]$. We will
refer to this set of numbers, along with the space dimension $n$,
as to the \textit{a-priori data}. 
\end{definition}

\begin{theorem}\label{teorema principale}
Let $\Omega$, $D_j$, $j=1,\dots , N$ and $\Sigma$ be a domain, $N$ subdomains of $\Omega$ and a portion of $\partial\Omega$ as in section \ref{subsec assumption domain} respectively.
If $\sigma_{A}^{(i)}\in\mathcal{C}$, $i=1,2$ are two conductivities of type

\begin{equation}\label{conduttivita anisotrope}
\sigma_{A}^{(i)}(x)=\sum_{j=1}^{N}\gamma_{j}^{(i)}\:A(x)\chi_{D_j}(x)\qquad
x\in\Omega,\:i=1,2,
\end{equation}

then we have

\begin{equation}\label{stabilita' globale}
||\sigma_{A}^{(1)}-\sigma_{A}^{(2)}||_{L^{\infty}(\Omega)}\leq C
||\Lambda^{\Sigma}_{\sigma_{A}^{(1)}}-\Lambda^{\Sigma}_{\sigma_{A}^{(2)}}||_{\ast},
\end{equation}

where $C$ is a positive constant that depends on the a-priori data
only.

\end{theorem}


\section{\normalsize{Proof of the main result}}\label{PMR}

The proof of our main result (theorem \ref{teorema principale}) is based on an argument that combines asymptotic type of estimates for the Green's function of the operator

\begin{equation}\label{operatore conduttivita' misurabile}
L=\mbox{div}\left(\sigma(x)\nabla\right)\qquad\mbox{in}\quad\Omega,
\end{equation}

(theorem \ref{teorema stime asintotiche}), with $\sigma\in\mathcal{C}$, together with a result of unique continuation (proposition \ref{proposizione unique continuation finale}) for solutions to

\[Lu=0,\qquad\mbox{in}\quad\Omega.\]

We shall give the precise formulation of these results in what
follows.

\subsection{\normalsize{Measurable conductivity $\sigma$}}

We shall start with some general considerations about the Green's
function $G(x,y)$ associated to the operator \eqref{operatore conduttivita' misurabile}, where $\sigma$ is merely a measurable matrix valued function
satisfying the ellipticity condition \eqref{ellitticita'sigma}.

\subsubsection{\normalsize{Green's function}}
 If $L$ is the operator given in \eqref{operatore conduttivita' misurabile}, then for every $y\in\Omega$, the Green's function $G(\cdot,y)$ is the weak solution to the Dirichlet problem

\setcounter{equation}{0}
\begin{equation}\label{GC}
\left\{
\begin{array}
{lcl} \mbox{div}(\sigma\nabla G(\cdot,y))=-\delta(\cdot - y)\ ,&
\mbox{in $\Omega$ ,}
\\
 G(\cdot,y)= 0\ ,   & \mbox{on $\partial\Omega$ ,}
\end{array}
\right.
\end{equation}

where $\delta(\cdot -y)$ is the Dirac measure at $y$. We recall
that $G$ satisfies the properties  (\cite{Lit-St-W})


\begin{equation}\label{simmetry G}
G(x,y)=G(y,x)\qquad\mbox{for}\:\mbox{every}\:x,y\in\Omega,\quad x\neq
y,
\end{equation}

\begin{eqnarray}\label{standardbeh}
0<G(x,y)<|x-y|^{2-n}\qquad\mbox{for}\:\mbox{every}\:x,y\in\Omega,\quad
x\neq y .
\end{eqnarray}

Moreover, the following result holds true.

\begin{proposition}\label{proposizione green function}
For any $y\in \Omega$ and every $r>0$ we have that
\begin{eqnarray}\label{caccio}
\int_{\Omega\setminus B_r(y)}|\nabla G(\cdot, y)|^2\le C r^{2-n} \
\end{eqnarray}
\end{proposition}
where $C>0$ depends on $\lambda$ and $n$ only.
\begin{proof}
The proof can be obtained by combining Caccioppoli inequality with
\eqref{standardbeh} (\cite{A-V}, Proposition 3.1).
\end{proof}

\subsubsection{\normalsize{Integral solutions of $L$}}

Let $\sigma^{(i)}$, $i=1,2$ be two measurable matrix valued
functions satisfying the ellipticity condition
\eqref{ellitticita'sigma} and let $G_{i}(x,y)$ be the Green's
functions associated to the operators

\begin{equation}\label{operatori Li}
L_i =
\mbox{div}\left(\sigma^{(i)}(x)\nabla\right)\qquad\mbox{in}\quad\Omega,\quad
i=1,2.
\end{equation}

Let $\mathcal{U}$ be an open subset of $\Omega$ and
$\mathcal{W}=\Omega\setminus\overline{\mathcal{U}}$. For any
$y,z\in\mathcal{W}$ we define

\begin{eqnarray}
S_{\mathcal{U}}({y},z)=\int_{\mathcal{U}}(\sigma^{(1)}(x)-\sigma^{(2)}(x))\nabla_x
G_1(x,{y})\cdot\nabla_x G_2(z,x)dx .
\end{eqnarray}

\begin{remark}
\begin{equation}\label{remark on SU}
|S_{\mathcal{U}}({y},z)|\leq C
||\sigma^{(1)}-\sigma^{(2)}||_{L^{\infty}(\Omega)}\left(d(y)d(z)\right)^{1-\frac{n}{2}},
\quad\mbox{for\:every}\:y,z\in\mathcal{W},
\end{equation}

where $d(y)=dist(y,\mathcal{U})$ and $C$ is a positive constant
depending on $\lambda, \bar{\gamma}$ and $n$ only.
\end{remark}

Observe that \eqref{remark on SU} is a straightforward consequence
of H\"older inequality and Proposition \ref{proposizione green
function}. We constructed in this way an integral function
$S_{\mathcal{U}}(\cdot,\cdot)$ on $\mathcal{W}\times\mathcal{W}$,
which is written in terms of the two Green's functions $G_1(\cdot,
y)$, $G_2(\cdot, z)$ of $L_1$, $L_2$ respectively;
$S_{\mathcal{U}}(\cdot,z)$, $S_{\mathcal{U}}(y,\cdot)$ are in turn
solutions for $L_1$, $L_2$ respectively on the complement part of
$\mathcal{U}$ in $\Omega$. More precisely we have

\begin{theorem}\label{teorema soluzioni SU}
For every $y,z\in W$ we have that $S_{\mathcal{U}}(\cdot,z),
S_{\mathcal{{U}}}(y,\cdot)\in H^1_{loc}(W)$ are weak solutions to
\begin{eqnarray}
\quad\textnormal{div} \left(\sigma^{(1)}(\cdot)\nabla
S_{\mathcal{U}}(\cdot,z)\right)=0 \ ,\ \textnormal{div}
\left(\sigma^{(2)}(\cdot) \nabla S_{\mathcal{U}}(y,\cdot)\right)=0 \
,\quad\mbox{in}\:\mathcal{W}.
\end{eqnarray}

\end{theorem}

\begin{proof}
The proof relies on differentiation under the integral sign arguments and the symmetry of $G_i, i=1,2$.

\end{proof}

\subsection{\normalsize{Conductivity $\sigma\in\mathcal{C}$}}



We shall denote with
\begin{eqnarray}\label{GC2}
\Gamma(x,y)=\frac{1}{(n-2)\omega_n}|x-y|^{2-n},
\end{eqnarray}
the fundamental solution of the Laplace operator (here $\omega _n/n$ denotes the volume of the unit ball in $\mathbb{R}^n$). If $D_i$, $i=1,\dots , N$ are the domains introduced in section
\ref{subsec assumption domain} and $L$ is the operator given by
\eqref{operatore conduttivita' misurabile}, with
$\sigma\in\mathcal{C}$, we shall give asymptotic estimates for the Green's
function of $L$, with respect to
\eqref{GC2} at the interfaces between the domains $D_i$, $i=1,\dots N$.
These estimates are given below. In what follows let $G$ be the Green's function associated to
the operator $L$ in $\Omega$.


\subsubsection{\normalsize{Green's function}}\label{subsection Green function}

\begin{theorem}\label{teorema stime asintotiche}({\bf{Asymptotic estimates}})
For every $l\in\{1, \dots, K-1 \}$, let $\nu(P_{l+1})$ denote the
unit exterior normal to $D_{j_{l+1}}$ at the point $P_{l+1}$.
There exist constants $\beta\in (0,\alpha)$ and $\bar{C}>1$ depending
on $\bar{\gamma}, \lambda, M, \alpha$ and $n$ only such that the following inequalities hold
true for every $\bar{x}\in B_{\frac{r_0}{\bar{C}}}(P_{l+1})\cap
D_{j_{l+1}}$ and every $\bar{y}=P_{l+1}+r\nu(P_{l+1})$, where $r\in
(0,\frac{r_0}{\bar{C}})$
\begin{eqnarray}\label{asyfun}
\left|G(\bar{x},\bar{y}) -
\frac{2}{\gamma_{j_l}+\gamma_{j_l+1}}\Gamma(J(\bar{x}),J(\bar{y}))\right|\le\frac{\bar{C}}{r_0^{\beta}}|\bar{x}-\bar{y}|^{\beta+2-n}
\ \ ,
\end{eqnarray}
\begin{eqnarray}\label{asygra}
\ \  \ \ \ \left|\nabla_x G(\bar{x},\bar{y}) -
\frac{2}{\gamma_{j_l}+\gamma_{j_l+1}}\nabla_x\Gamma(J(\bar{x}),J(\bar{y}))\right|\le\frac{\bar{C}}{r_0^{\beta}}|\bar{x}-\bar{y}|^{\beta+1-n}
\ \ ,
\end{eqnarray}

where $J$ is the positive definite matrix such that
$J=\sqrt{A(0)^{-1}}$\ .
\end{theorem}

\subsubsection{\normalsize{Integral solutions of $L$: unique continuation}}

We recall that
up to a rigid transformation of coordinates we can assume that

\[P_1=0\qquad ;\qquad (\mathbb{R}^n\setminus\Omega)\cap B_{r_0}=\{(x^{\prime},x_n)\in B_{r_{0}}\:|\:x_n <\varphi(x^{\prime})\},\]

where $\varphi$ is a Lipschitz function such that

\[\varphi(0)=0\qquad\textnormal{and}\qquad ||\varphi||_{C^{0,1}(B_{r_{0}}^{\prime})}\leq Lr_0.\]

Denoting by

\[D_0=\left\{x\in(\mathbb{R}^n\setminus\Omega)\cap B_{r_0}\:\bigg|\:|x_i|<\frac{2}{3}r_0,\:i=1,\dots , n-1,\:\left|x_n-\frac{r_0}{6}\right|<\frac{5}{6}r_0\right\},\]

it turns out that the augmented domain $\Omega_0=\Omega\cup D_0$
is of Lipschitz class with constants $\frac{r_0}{3}$ and
$\tilde{L}$, where $\tilde{L}$ depends on $L$ only. We consider
the operator $L_i$ given by \eqref{operatori Li} and extend
$\sigma^{(i)}\in\mathcal{C}$ to $\tilde{\sigma}^{(i)}=\tilde{\gamma}^{(i)}\:\tilde{A}$ on $\Omega_0$, by
setting $\tilde{\gamma}^{(i)}|_{D_0}=1$, and extending $A$ to $\tilde{A}\in C^{0,1}(\Omega_0)$ with Lipschitz constant $L$, for $i=1,2$. We denote by
$\tilde{G}_i$ the Green function associated to
$\tilde{L_i}=\textnormal{div}(\tilde\sigma^{(i)}(x)\nabla\cdot)$ in $\Omega_0$,
for $i=1,2$. For any number $r\in \left(0,\frac{2}{3}r_0\right)$
we also denote

\[(D_0)_r = \left\{x\in D_0\:|\:dist(x,\Omega)>r\right\}.\]

Let us fix $k\in\{2,\dots N\}$ and recall that there exist
$j_1,\dots j_K \in\{1,\dots N\}$ such that

\[D_{j_1}=D_1,\dots D_{j_K}=D_k.\]

We denote

\[\mathcal{W}_K = \bigcup_{i=0}^{K}D_{j_i},\qquad \mathcal{U}_k = \Omega_0\setminus\overline{\mathcal{W}_K},\quad\textnormal{when}\:k\geq 0\]

($D_{j_0}=D_0$) and for any $y,z\in\mathcal{W}_K$

\[\tilde{S}_{\mathcal{U}_K}(y,z)=\int_{\mathcal{U}_K}(\tilde\sigma^{(1)}_A-\tilde\sigma^{(2)}_A)\nabla\tilde{ G}_1(\cdot,y)\cdot\nabla\tilde{G}_2(\cdot,z),\quad \textnormal{when}\:k\geq 0.\]

We introduce for any number $b>0$ as in \cite{A-V}, the concave
non decreasing function $\omega_{b}(t)$, defined on $(0,+\infty)$,

\begin{displaymath}
\omega_{b}(t)=\left\{ \begin{array}{ll} 2^{b}e^{-2}|\log t|^{-b},
&\quad
t\in (0,e^{-2}),\\
e^{-2}, &\quad t\in[e^{-2},+\infty)
\end{array} \right.
\end{displaymath}

and denote

\[\omega_{b}^{(1)}=\omega,\qquad \omega_{b}^{(j)}=\omega_{b}\circ \omega_{b}^{(j-1)}.\]

The following parameters shall also be introduced

\begin{eqnarray*}
& &\beta=\arctan{\frac{1}{L}},\quad\beta_1 = \arctan{\left(\frac{\sin\beta}{4}\right)},\quad\lambda_1=\frac{r_0}{1+\sin\beta_1}\\
& & \rho_1=\lambda_1\sin\beta_1,\quad a=\frac{1-\sin\beta_1}{1+\sin\beta_1}\\
& & \lambda_k=a\lambda_{k-1},\quad \rho_k = a\rho_{k-1},\quad\textnormal{for\:every}\:k\geq 2,\\
& & d_k=\lambda_k-\rho_k,\quad k\geq 1.
\end{eqnarray*}

Let us denote here and in the sequel

\[E=||\sigma^{(1)}_{A}-\sigma^{(2)}_{A}||_{L^{\infty}(\Omega)}.\]

The following estimate for
$\tilde{S}_{\mathcal{U}_K}(y,z)$ holds true.

\begin{proposition}\label{proposizione unique continuation finale}({\bf{Estimates of unique continuation}})
If, for a positive number $\varepsilon_0$, we have

\begin{equation}\label{ip S tilde}
\left|\tilde{S}_{\mathcal{U}_K}(y,z)\right|\leq
r_0^{2-n}\varepsilon_0,\quad for\:every\: (y,z)\in
(D_0)_{\frac{r_0}{3}}\times(D_0)_{\frac{r_0}{3}},
\end{equation}

then the following inequality holds true for every $r\in (0,d_1]$

\begin{equation}
\left|\tilde{S}_{\mathcal{U}_K}\left(w_{\bar{h}}(P_{K+1}),w_{\bar{h}}(P_{K+1})\right)\right|
\leq
r_0^{2-n}C^{\bar{h}}(E+\varepsilon_0)\left(\omega_{1/C}^{(2K)}\left(\frac{\varepsilon_0}{E+\varepsilon_0}\right)\right)
^{\left(1/C\right)^{\bar{h}}},
\end{equation}

where $P_{K+1}\in\Sigma_{K+1}$,
$\bar{h}=min\{k\in\mathbb{N}\:|\:d_k\leq r\}$,
$w_{\bar{h}}(P_{K+1})=P_{K+1}-\lambda_{\bar{h}}\nu(P_{K+1})$,
$\nu$ is the exterior unit normal to $\partial{D}_K$ and $C\geq 1$
depends on the a-priori data only.

\end{proposition}

\subsection{\normalsize{Proof of Theorem \ref{teorema principale}}}

\begin{proof}{\textit{Proof of Theorem \ref{teorema principale}.}} We denote by $\Lambda_i$ the map
$\Lambda^{(\Sigma)}_{\sigma^{(i)}_A}$, for $i=1,2$ and, for every
$k\in\{0,\dots , K\}$, the subscript $j_k$ will be replaced by
$k$. This should simplify the notation. Let us point out that

\[||(\sigma_{A}^{(1)}-\sigma_{A}^{(2)})||_{L^{\infty}(\Omega)}\leq
\bar{A}||\gamma^{(1)}-\gamma^{(2)}||_{L^{\infty}(\Omega)},\]

where

\[\gamma^{(i)}=\sum_{j=1}^{N} \gamma_{j}^{(i)}\chi_{D_j}(x),\qquad i=1,2,\]

therefore \eqref{stabilita' globale} trivially follows from

\begin{equation}\label{stabilita'globale gamma}
||\gamma^{(1)}-\gamma^{(2)}||_{L^{\infty}(\Omega)}\leq C||\Lambda_1
-
\Lambda_2||_{\mathcal{L}(H^{1/2}_{co}(\Sigma),H^{-1/2}_{co}(\Sigma))}
\end{equation}

which we shall prove. Moreover we shall denote

\[\varepsilon=||\Lambda_1 - \Lambda_2||,\qquad \delta_k=||\tilde\gamma^{(1)}-\tilde\gamma^{(2)}||_{L^{\infty}(\mathcal{W}_k)}.\]

We start by recalling that for every $y,z\in
D_0$ we have

\[\left<(\Lambda_1 - \Lambda_2)\tilde{G}_1(\cdot,y),\tilde{G}_2(\cdot,z)\right>
=\int_{\Omega}(\tilde\gamma^{(1)}-\tilde\gamma^{(2)})A(\cdot)\nabla\tilde{G}_1(\cdot,y)\cdot\nabla\tilde{G}_2(\cdot,z)\]

and that, for every $k\in\{1,\dots K\}$,

\[\tilde{S}_{\mathcal{U}_{k-1}}(y,z)=\int_{\mathcal{U}_{k-1}}
(\tilde\gamma^{(1)}-\tilde\gamma^{(2)})\tilde{A}(\cdot)
\nabla\tilde{G}_1(\cdot,y)\cdot\nabla\tilde{G}_2(\cdot,z),\]

therefore

\begin{eqnarray}\label{stima S}
|\tilde{S}_{\mathcal{U}_{k-1}}(y,z)| &\leq & \varepsilon
||\tilde{G}_1(\cdot,y)||_{H^{1/2}_{co}(\Sigma)}||\tilde{G}_2(\cdot,z)||_{H^{1/2}_{co}(\Sigma)}\nonumber\\
&+&\delta_{k-1}\bar{A}||\nabla\tilde{G}_1(\cdot,y)||_{L^2(\mathcal{W}_{k-1})}
||\nabla\tilde{G}_2(\cdot,z)||_{L^2(\mathcal{W}_{k-1})}\nonumber\\
&\leq & C (\varepsilon+\delta_{k-1})r_0^{2-n},\qquad
\textnormal{for\:every}\:y,z\in(D_0)_{r_0/3},
\end{eqnarray}

where $C$ depends on $A$, $L$, $\lambda$, $\bar{A}$ and $n$. Let $\rho_0=\frac{r_0}{\bar{C}}$, where $\bar{C}$ is the constant introduced in Theorem \ref{teorema stime asintotiche}, let $r\in(0,d_2)$ and denote

\[w=P_k+\sigma\nu,\qquad\textnormal{where}\:\sigma=a^{\bar{h}-1}\lambda_1,\]


then

\begin{equation}\label{S=I1+I2}
\tilde{S}_{\mathcal{U}_{k-1}}(\omega,\omega)=I_1(\omega)+I_2(\omega),
\end{equation}

where

\[I_1(\omega)=\int_{B_{\rho_0}(P_k)\cap D_k}(\gamma^{(1)}-\gamma^{(2)})A(\cdot)\nabla\tilde{G}_1(\cdot,\omega)\cdot
\nabla\tilde{G}_1(\cdot,\omega),\]

\[I_2(\omega)=\int_{\mathcal{U}_{k-1}\setminus (B_{\rho_0}(P_k)\cap
D_k)}(\gamma^{(1)}-\gamma^{(2)})A(\cdot)
\nabla\tilde{G}_1(\cdot,\omega)\cdot
\nabla\tilde{G}_1(\cdot,\omega)\]

and (see \cite{A-V})

\begin{equation}\label{stima I2}
|I_2(\omega)|\leq CE\rho_0^{2-n},
\end{equation}

where $C$ depends on $\lambda$, $\bar{A}$ and $n$ only. To estimate
$I_1(\omega)$ we recall Theorem \ref{teorema stime asintotiche}
which leads to



\begin{eqnarray*}
|I_1(\omega)| &\geq &
|\gamma_k^{(1)}-\gamma_k^{(2)}|C_1\int_{B_{\rho_0}(P_k)\cap
D_k}|\nabla_x\Gamma(Jx,J\omega)|^2\noindent\\
&-&C_2\int_{B_{\rho_0}(P_k)\cap
D_k}|A(x)||\nabla_x\Gamma(Jx,J\omega)|\frac{|x-\omega|^{1-n+\beta}}{\rho_0^{\beta}}\noindent\\
&-&C_3\int_{B_{\rho_0}(P_k)\cap
D_k}|A(x)|\frac{|x-\omega|^{2-2n+\beta}}{\rho_0^{2\beta}},
\end{eqnarray*}

where $C_1,C_2,C_3$ are constants that depends on
$M,\lambda,\alpha$, $\bar{A}$ and $n$ only. Therefore, by combining
\eqref{stima S} together with \eqref{S=I1+I2} and \eqref{stima
I2}, we obtain

\begin{eqnarray*}
|I_1(\omega)| &\geq &
|\gamma_k^{(1)}-\gamma_k^{(2)}|C_1\int_{B_{\rho_0}(P_k)\cap
D_k}\frac{|J^2(x-\omega)|^2}{|J(x-\omega)|^{2n}}\noindent\\
&-&\frac{C_2E}{{\rho_0^{\beta}}}\int_{B_{\rho_0}(P_k)\cap
D_k}\frac{|J^2(x-\omega)|}{|J(x-\omega)|^{n}}|x-\omega|^{1-n+\beta}\noindent\\
&-&\frac{C_3E}{{\rho_0^{2\beta}}}\int_{B_{\rho_0}(P_k)\cap
D_k}|x-\omega|^{2(1-n)+\beta}.
\end{eqnarray*}

Therefore

\begin{eqnarray}
|I_1(\omega)| &\geq &
|\gamma_k^{(1)}-\gamma_k^{(2)}|C_1\int_{B_{\rho_0}(P_k)\cap
D_k}|x-\omega|^{2(1-n)}\noindent\\
&-&\frac{C_2E}{{\rho_0^{\beta}}}\int_{B_{\rho_0}(P_k)\cap
D_k}|x-\omega|^{2(1-n)+\beta}\noindent\\
&-&\frac{C_3E}{{\rho_0^{2\beta}}}\int_{B_{\rho_0}(P_k)\cap
D_k}|x-\omega|^{2(1-n+\beta)},
\end{eqnarray}

which leads to

\begin{equation}\label{stima I1 dal basso}
|I_1(\omega)|\geq
C_1|\gamma_k^{(1)}-\gamma_k^{(2)}|\sigma^{2-n}-C_2E\frac{\sigma^{2-n+\beta}}{\rho_0^{\beta}},
\end{equation}

where $\beta$ is the number introduced in Theorem \ref{teorema
stime asintotiche} and $C_1$, $C_2$ depend on $M,\lambda,\alpha$, $\bar{A}$
and $n$ only. By combining \eqref{stima I1 dal basso} together with
\eqref{S=I1+I2} and \eqref{stima I2} we obtain

\begin{equation}\label{stima gamma1- gamma2}
C_1|\gamma_k^{(1)}-\gamma_k^{(2)}|\sigma^{2-n}\leq
|\tilde{S}_{\mathcal{U}_{k-1}}(\omega,\omega)|+C_2E\frac{\sigma^{2-n+\beta}}{\rho_0^{\beta}}
\end{equation}

and by Proposition \ref{proposizione unique continuation finale} and \eqref{stima S} we obtain

\[
|\tilde{S}_{\mathcal{U}_{k-1}}(\omega,\omega)| \leq  \sigma ^{2-n}C^{\bar{h}}(E+\varepsilon +\delta_{k-1})
\bigg(\omega_{\frac{1}{C}}\Big(\frac{\varepsilon+\delta_{k-1}}{E+\varepsilon+\delta_{k-1}}\Big)\bigg)^
{\big(\frac{1}{C}\big)^{\bar{h}}},
\]

where $C\geq 1$ is a constant depending on $A$, $L$, $\bar{A}$, $M$, $N$, $\alpha$, $\lambda$ and $n$ only, therefore

\begin{equation}
|\gamma_k^{(1)}-\gamma_k^{(2)}|\leq
C^{\bar{h}}(\varepsilon+\delta_{k-1}+E)\left(\omega_{1/C}^{(2(k-1))}\right)
^{\left(\frac{1}{C}\right)^{\bar{h}}}+C_2 E\Big(\frac{\sigma}{\rho_0}\Big)^{\theta}.
\end{equation}

We need to estimate $C^{\bar{h}}$ and $\Big(\frac{1}{C}\Big)^{\bar{h}}$, where $C>1$. It turns out that

\begin{eqnarray}
C^{\bar{h}} &\leq & C^2\Big(\frac{d_1}{r}\Big)^{-\frac{1}{\log_c a}}\nonumber\\
\big(\frac{1}{C}\Big)^{\bar{h}} &\leq & \Big(\frac{1}{C}\Big)^2\Big(\frac{r}{d_1}\Big)^{-\frac{1}{\log_c a}},
\end{eqnarray}

therefore

\begin{equation}\label{461}
|\gamma_k^{(1)}-\gamma_k^{(2)}|\leq
C(\varepsilon+\delta_{k-1}+E)\left(\left(\frac{d_1}{r}\right)^C\:\left(\omega_{1/C}^{(2(k-1))}\right)
^{\left(\frac{r}{d_1}\right)^C}+\left(\frac{r}{d_1}\right)^{\theta}\right).
\end{equation}

By \eqref{461} we obtain for every $k\in\{1,\dots , K\}$

\[\delta_{k}\leq \delta_{k-1}+C(\varepsilon+\delta_{k-1}+E)\left(\omega_{1/C}^{(2(k+1))}\left(\frac{\varepsilon+\delta_{k-1}}{\varepsilon+\delta_{k-1}+E}\right)\right)^
{\frac{1}{C}},
\]

which leads to

\begin{equation*}
||\gamma^{(1)}-\gamma^{(2)}||_{L^{\infty}(\Omega)}\leq C(\varepsilon + E)\left(\omega_{\frac{1}{C}}^{(K^2)}\left(\frac{\varepsilon}{\varepsilon + E}\right)\right)^{\frac{1}{C}},
\end{equation*}

therefore

\begin{equation}\label{463}
E\leq C(\varepsilon +E)\left(\omega_{\frac{1}{C}}^{(K^2)}\left(\frac{\varepsilon}{\varepsilon + E}\right)\right)^{\frac{1}{C}}.
\end{equation}

Assuming that $E>\varepsilon e^2$ (if this is not the case then the theorem is proven) we obtain

\[E\leq C \left(\frac{E}{e^2}+E\right)\left(\omega_{\frac{1}{C}}^{(K^2)}\left(\frac{\varepsilon}{E}\right)\right)^{\frac{1}{C}},\]

which leads to

\[\frac{1}{C}\leq \omega_{\frac{1}{C}}^{(K^2)}\left(\frac{\varepsilon}{E}\right)\]

therefore

\[E\leq \frac{1}{\omega_{\frac{1}{C}}^{(-K^2)}\left(\frac{1}{C}\right)}\:\varepsilon,\]

which concludes the proof.

\end{proof}




\section{\normalsize{Proof of technical propositions}}\label{PP}

\subsection{\normalsize{Proof of the asymptotic estimates}}\label{AE}

\setcounter{equation}{0}

Whenever $\varphi$ is a Lipschitz continuous function on $\mathbb{R}^{n-1}$, we shall denote by $Q^+_{\varphi,r}$ and
$Q^-_{\varphi,r}$ the following sets

\begin{eqnarray}
Q^+_{\varphi,r}=\{(x',x_n)\in Q_r \ | x_n>\varphi(x') \}\ ,\\
Q^-_{\varphi,r}=\{(x',x_n)\in Q_r \ | x_n<\varphi(x') \} \ .
\end{eqnarray}

Let $0<\mu<1$  and $B^+\in C^{\mu}(\overline{Q^+_{\varphi,r}})$, $B^-\in
C^{\mu}(\overline{Q^-_{\varphi,r}})$ be symmetric,
positive definite matrix valued functions and define

\begin{displaymath}
B(x)=\left\{ \begin{array}{ll}B^+(x),
&\quad
x\in Q^+_{\varphi,r},\\
B^-(x), &\quad x\in Q^-_{\varphi,r}
\end{array} \right.
\end{displaymath}

such that $B$ satisfies the uniform ellipticity condition
\begin{eqnarray}\label{unifellip2}
{\lambda_0}^{-1}|\xi|^2\le B(x)\xi\cdot\xi\le {\lambda_0}|\xi|^2\ ,\ &&\mbox{for almost every}\ x\in Q_r, \\
 &&\mbox{for every}\ \xi\in\mathbb{R}^n \ . \nonumber
\end{eqnarray}

where $\lambda_0>0$ is a constant.

\begin{theorem}\label{rego}
Let $k>0,r>0$ and $0<\alpha<1$ be fixed numbers. Moreover, let $B$
be a matrix as above. Assume that $\varphi\in C^{1,\alpha}(B'_r)$ and let $U\in H^1(Q_r)$ be a solution to

\begin{eqnarray}
div \left(\left(1+(k-1)\chi_{Q^+_{\varphi,r}}\right)B\nabla U\right)=0 \ .
\end{eqnarray}

Suppose $\alpha'$ satisfies at the same time  $0<\alpha'\le\mu$
and $\alpha'<\frac{\alpha}{(\alpha+1)n}$. Then, there exists a
positive constant $C$ such that for any $\rho\le \frac{r}{2}$ and
for any $x\in Q_{r-2\rho}$, the following estimate holds
\begin{eqnarray}\label{stimareg}
&&\|\nabla U\|_{L^{\infty}(Q_{\rho}(x))}+ {\rho}^{{\alpha}'}|\nabla U|_{\alpha',Q_{\rho}(x)\cap Q^+_{\varphi,r} } + {\rho}^{{\alpha}'}|\nabla U|_{\alpha',Q_{\rho}(x)\cap Q^-_{\varphi,r} } \nonumber \\
&&\le \frac{C}{\rho^{1+n/2}}\|U\|_{L^2(Q_{2\rho}(x))}\ ,
\end{eqnarray}
where $C$ depends on $\|\varphi\|_{C^{1,\alpha}(B_r')},
k,\alpha,\alpha',n, \lambda_0$,
$\|B^+\|_{C^{\alpha'}(\overline{Q^+_{\varphi,r}})}$ and
$\|B^-\|_{C^{\alpha'}(\overline{Q^-_{\varphi,r}})}$ only.

\end{theorem}

\begin{proof}
For the proof we refer to \cite[Theorem 1.1]{Li-Vo}, where the
authors, among various results, obtain piecewise $C^{1,\alpha'}$
estimates for solutions to divergence form elliptic equations with
piecewise H\"{o}lder continuous coefficients (see also
\cite{Li-Ni}).
\end{proof}

We fix $l\in \{1, \dots, K-1 \}$. There exists a rigid
transformation of coordinates under which $P_{l+1}=0$ and

\begin{eqnarray}
&&\Sigma_l\cap Q_{\frac{r_0}{3}}=\{x\in Q_{\frac{r_0}{3}} | x_n=\varphi(x') \} \ ,\\
&&D_{j_{l+1}}\cap Q_{\frac{r_0}{3}}=\{x\in Q_{\frac{r_0}{3}} | x_n>\varphi(x') \} \ ,\\
&&D_{j_{l}}\cap Q_{\frac{r_0}{3}}=\{x\in Q_{\frac{r_0}{3}} |
x_n<\varphi(x') \} \ ,
\end{eqnarray}
where $\varphi$ is a $C^{1,\alpha}$ function on
$B'_{\frac{r_0}{3}}$ satisfying

\begin{eqnarray}
\varphi(0)=|\nabla \varphi (0)|=0 \ , \ \ \ \ \
\|\varphi\|_{C^{1,\alpha}}(B'_{r_0})\le Mr_0\ .
\end{eqnarray}

Moreover, up to a possible replacement of $\gamma$ with
$\frac{\gamma}{\gamma_{j_l}}$, we can assume that
$\gamma|_{D_{jl}}=1$ and $\gamma|_{D_{j_{l+1}}}=k$ where $k$ is a
real number which satisfies
\begin{eqnarray}
\bar{\gamma}\le k \le\bar{\gamma}^{-2} \ .
\end{eqnarray}

Let $\tau$ be a $C^{\infty}$ function on $\mathbb{R}$ such that
$0\le \tau\le 1, \tau(s)=1$ for every $s\in (-1,1), \tau(s)=0$ for
every $s\in \mathbb{R}\setminus(-2,2)$ and $|\tau'(s)|\le 2$ for
every $s\in \mathbb{R}$.

We introduce
\begin{eqnarray}
r_1=\frac{r_0}{3}\min\left\{\frac{1}{2}(8M)^{-\frac{1}{\alpha}},\frac{1}{4}\right\}
\end{eqnarray}
and we consider the following change of variable $\xi=\Phi(x)$
given by

\begin{equation}\label{diffeo}
\left\{
\begin{array}
{lcl} \xi'=x' \ ,
\\
 \xi_n=x_n - \varphi(x')\tau(\frac{|x'|}{r_1})\tau(\frac{x_n}{r_1}) \ .
\end{array}
\right.
\end{equation}

It can be verified that the map $\Phi$ is a
$C^{1,\alpha}(\mathbb{R}^n,\mathbb{R}^n)$ and it satisfies the
following properties

\begin{eqnarray}
\Phi(\Sigma_l\cap Q_{r_1})=\{x\in Q_{r_1} \ | x_n=0  \} \ ,
\end{eqnarray}

\begin{eqnarray}
\Phi(x)=x, \quad \mbox{for every}\ x\in \mathbb{R}^n\setminus Q_{2r_1}
\ ,
\end{eqnarray}

\begin{eqnarray}\label{bili}
C^{-1}|x_1-x_2| &\le & |\Phi(x_1)-\Phi(x_2)|\nonumber\\
&\le & C|x_1-x_2|, \quad\mbox{for every} \ x_1,x_2 \in \mathbb{R}^n ,
\end{eqnarray}

\begin{eqnarray}\label{stimah}
|\Phi(x)-x|\le \frac{C}{r_0^{\alpha}}|x|^{1+\alpha},\quad\mbox{and}
\end{eqnarray}

\begin{eqnarray}
|D\Phi(x)-I|\le \frac{C}{r_0^{\alpha}}|x|^{\alpha},\quad \mbox{for
every}\ \ x\in \mathbb{R}^n \ ,
\end{eqnarray}
where $C,C>1$, depends on $M$ and $\alpha$ only and $I$ denotes the
identity matrix.

Let $y_n\in (-\frac{r_1}{2},0)$ and $y=ye_n$.  We set

\begin{eqnarray}
\eta=\Phi(y)\ ,
\end{eqnarray}
\begin{eqnarray}
\tilde{G}(\xi, \eta)= G(\Phi^{-1}(\xi),\Phi^{-1}(\eta))\ ,
\end{eqnarray}
\begin{eqnarray}
J(\xi)=(D\Phi)(\Phi^{-1}(\xi))\ ,
\end{eqnarray}
\begin{eqnarray}
\tilde{\sigma}(\xi)=\frac{1}{\mbox{det}J(\xi)}J(\xi)\gamma(\Phi^{-1}(\xi))A(\Phi^{-1}(\xi))(J(\xi))^T
\ ,
\end{eqnarray}
we have that $\tilde{G}(\cdot, \eta)$ is a solution to
\begin{equation}\label{GCt}
\left\{
\begin{array}
{lcl} \mbox{div}( \tilde{\sigma}\nabla(\xi)
\tilde{G}(\cdot,\eta))=-\delta(\cdot - \eta)\ ,& \mbox{in $\Omega$
,}
\\
 \tilde{G}(\cdot,\eta)= 0\ ,   & \mbox{on $\partial\Omega$ .}
\end{array}
\right.
\end{equation}


We have
\begin{eqnarray}\label{at}
\tilde{\sigma}(\xi)=(1 + (k-1)\chi^+(\xi))B(\xi), \quad \mbox{for any} \ \xi
\in Q_{r_1}\ ,
\end{eqnarray}
where $\chi^+$ is the characteristic function of $\mathbb{R}_+^n$
and
\begin{eqnarray}
B(\xi)=\frac{1}{\mbox{det}J(\xi)}J(\xi)A(\Phi^{-1}(\xi))(J(\xi))^T
.
\end{eqnarray}

Furthermore, we have that $B$ is of class $C^{\alpha}$ and
\begin{eqnarray}\label{bh}
\|B\|_{C^{0,\alpha}(\Omega)}\le C \ ,
\end{eqnarray}

where $C>0$ is a constant depending on $M,\alpha, \lambda,
\bar{A}$ only. We also have that
$B(0)={A}(0)$ \ . We denote

\begin{eqnarray}\label{azero}
{\sigma}_0(\xi)=(1 + (k-1)\chi^+(\xi))A(\xi)
\end{eqnarray}

and with

\begin{eqnarray}\label{azeroi}
{\sigma}_{0,0}(\xi)=(1 + (k-1)\chi^+(\xi))A(0)
\end{eqnarray}

and we refer to $G_0$ as to the Green function solution to

\begin{equation}\label{GCz}
\left\{
\begin{array}
{lcl} \mbox{div}({\sigma}_{0,0}(\cdot)\nabla G_0(\cdot,y))=-\delta(\cdot -
y)\ ,& \mbox{in $\Omega$ ,}
\\
 G_0(\cdot,y)= 0\ ,   & \mbox{on $\partial\Omega$ .}
\end{array}
\right.
\end{equation}

We then define
\begin{eqnarray}\label{diff}
R(\xi, \eta)= \tilde{G}(\xi,\eta) - G_0(\xi,\eta)\ .
\end{eqnarray}

\begin{lemma}\label{lemmaasi}
For every $\xi\in B^+_{\frac{r_1}{4}}$ and $\eta_n\in
(-\frac{r_1}{4},0)$ we have that

\begin{eqnarray}
|R(\xi,e_n\eta_n)| + |\xi-e_n\eta_n||\nabla_{\xi}R(\xi,\eta)|\le
\frac{c}{r_1^{\beta}}|\xi-e_n\eta_n|^{\beta+2-n}\ ,
\end{eqnarray}

where $\beta\in (0,\alpha^2]$ depends on $\alpha$ and $n$ only and
$C$ depends on $M,\bar{\gamma}, \lambda, \bar{A}$ only.
\end{lemma}

\begin{proof}
It is easy to check that $R$ in \eqref{diff} satisfies

\begin{equation}\label{eqdiff}
\left\{
\begin{array}
{lcl} \mbox{div}_{\xi}(\tilde{\sigma}(\cdot)\nabla_{\xi}
R(\cdot,\eta))=-\mbox{div}_{\xi}((\tilde{\sigma}(\cdot)-{\sigma}_{0,0}(\cdot))\nabla_{\xi}
G_0(\cdot,\eta))\ ,& \mbox{in $\Omega$ ,}
\\
 R(\cdot,\eta)= 0\ ,   & \mbox{on $\partial\Omega$ .}
\end{array}
\right.
\end{equation}

By the representation formula over $\Omega$ we have that $R$ in
\eqref{diff} satisfies

\begin{eqnarray}
R(\xi,\eta)=\int_{\Omega}(\tilde{\sigma}(\zeta)-{\sigma}_{0,0}(\zeta))\nabla_{\zeta}G_0(\zeta,\eta)\cdot\nabla\tilde{G}(\zeta,\xi)d\zeta\
.
\end{eqnarray}
We consider $\xi\in Q^+_{\frac{r_1}{2}}$ and $\eta=e_n\eta_n$ and
we split $R$ as the sum of the following integrals
\begin{eqnarray}
&& R_1(\xi,\eta)=\int_{\Omega\setminus Q_{r_1}}(\tilde{\sigma}(\zeta)-{\sigma}_{0,0}(\zeta))\nabla_{\zeta}G_0(\zeta,\eta)\cdot\nabla\tilde{G}(\zeta,\xi)d\zeta\ \ ,\\
&&
R_2(\xi,\eta)=\int_{Q_{r_1}}(\tilde{\sigma}(\zeta)-{\sigma}_{0,0}(\zeta))\nabla_{\zeta}G_0(\zeta,\eta)\cdot\nabla\tilde{G}(\zeta,\xi)d\zeta\
\ .
\end{eqnarray}

By the bounds \eqref{apriorigamma},
\eqref{unifellip},\eqref{aprioria} and by combining the Schwartz
inequality with the Caccioppoli inequality we get

\begin{eqnarray}
|R_1(\xi,\eta)|\le \frac{C}{r_1^2}\|G_0(\cdot,
\eta)\|_{L^2(\Omega\setminus Q_{3r_1/{4}})}\|\tilde{G}(\cdot,
\eta)\|_{L^2(\Omega\setminus Q_{{3r_1}/{4}})},
\end{eqnarray}

where $C>0$ depends on $M,\alpha,\bar{\gamma},\lambda$ and
$\bar{A}$ only. By the standard behaviour \eqref{standardbeh} of
the Green functions at hand, it follows that

\begin{eqnarray}\label{R1b}
|R_1(\xi,\eta)|\le C r_1^{2-n}\ ,
\end{eqnarray}

where $C>0$ depends on $M,\alpha,\bar{\gamma},\lambda$ and
$\bar{A}$ only. Moreover being $B(0)=A(0)$, it follows that \eqref{aprioria} and \eqref{bh} lead to

\begin{eqnarray}\label{as1}
|\tilde{\sigma}(\xi)- {\sigma}_{0,0}(\xi)|\le \max\{1,k\}(|B(\xi)-A(0)| )\le \frac{C}{r_1}^{\alpha} |\xi|^{\alpha},
\end{eqnarray}

for any $\xi\in Q_{r_1}$, where $C$ depends on $M,\alpha, \bar{A}$
and $\bar{\gamma}$ only. Moreover, by \eqref{standardbeh} and by Theorem \ref{rego} we have
that

\begin{eqnarray}\label{as2}
|\nabla_{\zeta}G_0(\zeta,\xi)|\le C|\zeta-\xi|^{1-n} \ , \quad
\mbox{for every}\ \zeta, \xi\in Q_{r_1},
\end{eqnarray}

where $C$ depends on $M,\alpha, \bar{A}$ and $\bar{\gamma}$ only. By \eqref{bili} and the same arguments used above, we infer that

\begin{eqnarray}\label{as3}
|\nabla_{\zeta}\tilde{G}(\zeta,\xi)|\le C|\zeta-\xi|^{1-n}, \quad
\mbox{for every}\ \zeta, \xi\in Q_{r_1},
\end{eqnarray}

where $C$ depends on $M,\alpha, \bar{A}$ and $\bar{\gamma}$ only. We denote

\begin{eqnarray}
I_1=\int_{B_{4h}}|\zeta|^{\alpha}|\zeta-\xi|^{1-n}|\zeta-\eta|^{1-n}d\eta
\end{eqnarray}

and

\begin{eqnarray}
I_2=\int_{\mathbb{R}^n\setminus
B_{4h}}|\zeta|^{\alpha}|\zeta-\xi|^{1-n}|\zeta-\eta|^{1-n}d\eta .
\end{eqnarray}

By \eqref{as1}, \eqref{as2} and \eqref{as3} we have that

\begin{eqnarray}\label{R2b}
|R_2(\xi,\eta)|\le \frac{C}{r_1^{\alpha}}(I_1+I_2) \ .
\end{eqnarray}

Let us denote now $h=|\xi-\eta|$ and consider the following change of variables $\zeta=hw$; we set $t=\frac{\xi}{h}$ and $s=\frac{\eta}{h}$ and it follows that for any $t,s\in \mathbb{R}^n$ we have that
$|t-s|=1$. We obtain that

\begin{eqnarray}
I_1\le 4^{\alpha}h^{\alpha+2-n}\int_{B_4}|t-w|^{1-n}|s-w|^{1-n}dw\
.
\end{eqnarray}

Let us now set
\begin{eqnarray}
F(t,s)= \int_{B_4}|t-w|^{1-n}|s-w|^{1-n}dw\ \ .
\end{eqnarray}
From standard bounds (see for instance, \cite[Chapter 2]{Mi}) we
have that

\begin{eqnarray}
F(t,s)\le C ,
\end{eqnarray}
where $C$ depends on $n$ only. Hence
\begin{eqnarray}\label{i1}
I_1\le C h^{\alpha+2-n}\ ,
\end{eqnarray}
where $C$ depends on $n$ only. We consider now integral $I_2$. We recall that $\eta=e_n\eta_n$, where
$\eta_n\in(-\frac{r_1}{2},0)$ and $\xi\in Q^+_{\frac{r_1}{2}}$, hence we
have

\begin{eqnarray}
|\eta|=-\eta_n\le -\eta_n +\xi_n\le |\xi-\eta|=h \ ,
\end{eqnarray}

which leads to

\begin{eqnarray}
|\xi|\le |\xi-\eta| + |\eta|\le 2h \ .
\end{eqnarray}

On the other hand, we have that for any $\zeta\in
\mathbb{R}^n\setminus B_{4h}$

\begin{eqnarray}
|\zeta|\le |\zeta-\eta|+|\eta|\le |\zeta-\eta|+\frac{1}{4}|\zeta |
\ ,
\end{eqnarray}

hence we get

\begin{eqnarray}\label{esterno1}
\frac{3}{4}|\zeta|\le |\zeta-\eta|\ .
\end{eqnarray}

and by using the same arguments we get

\begin{eqnarray}\label{esterno2}
\frac{1}{2}|\zeta|\le |\xi-\zeta|\ , \ \mbox{for any}\ \zeta \in
\mathbb{R}^n\setminus B_{4h}\ .
\end{eqnarray}

By combining \eqref{esterno1} together with \eqref{esterno2}, we obtain that
\begin{eqnarray}\label{i2}
I_2\le \left(\frac{8}{3}\right)^{1-n}\int_{\mathbb{R}^n\setminus
B_{4h}}|\zeta|^{\alpha+2-2n}d\zeta \le C h^{\alpha+2-n} \ ,
\end{eqnarray}

where $C$ depends on $\alpha$ and $n$ only. By combining \eqref{R1b},\eqref{R2b},\eqref{i1} and \eqref{i2} we
obtain

\begin{eqnarray}\label{stimaR}
|R(\xi,\eta)|\le \frac{C}{r_1^{\alpha}}h^{\alpha+2-n}\ ,
\end{eqnarray}

where $C$ depends on $M,\alpha, \bar{A},n$ and $\bar{\gamma}$ only. Let us fix $\xi \in B^+_{\frac{r_1}{4}}$ and $\eta_n\in (-r_1/4,
0)$ and consider the cylinder

\begin{eqnarray}
Q=B'_{\frac{h}{8}}(\xi') \times \left(\xi_n, \xi_n
+\frac{h}{8}\right) \ .
\end{eqnarray}

Observing that $h=|\xi-(0,\eta_n e_n)|\le \frac{r_1}{2}$ we deduce
that $Q\subset Q^{+}_{\frac{r_1}{2}}$. Moreover
$Q\subset Q_{\frac{h}{4}(\xi)}$ and $\xi\in \partial Q$, then by choosing for instance $\alpha'=\frac{1}{2}\min \left \{\alpha,
\frac{\alpha}{(\alpha +1)n} \right \}$ in the statement of Theorem
\ref{rego} and observing that $(0,\eta_n e_n) \notin
Q_{\frac{h}{2}}(\xi)$, by  \eqref{stimareg} we obtain the
following bound for the seminorm

\begin{eqnarray}\label{seminorma1a}
|\nabla _{\xi}\tilde{G}(\cdot, e_n \eta_n)|_{\alpha', Q}&&\le |\nabla _{\xi}\tilde{G}(\cdot, e_n \eta_n)|_{\alpha', Q_{\frac{h}{4}(\xi)}\cap Q^+_{\frac{r_1}{2}}} \nonumber \\
&& \le C h^{-\alpha'-1-n/2}\|\nabla _{\xi}\tilde{G}(\cdot, e_n
\eta_n)\|_{L^2(Q_{\frac{h}{2}(\xi)})},
\end{eqnarray}

where $C$ depends on $M,\alpha, \bar{A},n$ and $\bar{\gamma}$ only. Furthermore by observing that for any $\tilde{\xi}\in
Q_{\frac{h}{2}(\xi)}$ we have that $|\tilde{\xi}-(0,e_n\eta_n)|\ge
\frac{h}{2}$ and by \eqref{standardbeh} we have that

\begin{eqnarray}\label{seminorma1}
|\nabla _{\xi}\tilde{G}(\cdot, e_n \eta_n)|_{\alpha', Q}\le C
h^{\alpha' +1-n}\ ,
\end{eqnarray}

where $C$ depends on $M,\alpha, \bar{A},n$ and $\bar{\gamma}$ only. By analogous argument we may also infer that

\begin{eqnarray}\label{seminorma2}
|\nabla _{\xi}{G_0}(\cdot, e_n \eta_n)|_{\alpha', Q}\le C
h^{\alpha' +1-n}\ ,
\end{eqnarray}

where $C$ depends on $M,\alpha, \bar{A},n$ and $\bar{\gamma}$ only. Hence by \eqref{diff}, \eqref{seminorma1} and \eqref{seminorma2}
we obtain

\begin{eqnarray}
|\nabla _{\xi}R(\cdot, e_n \eta_n)|_{\alpha', Q}\le C  h^{\alpha'
+1-n}\ ,
\end{eqnarray}

where $C$ depends on $M,\alpha, \bar{A},n$ and $\bar{\gamma}$
only. We recall the following interpolation inequality (see for instance
\cite[Proposition 8.3]{A-S})

\begin{eqnarray}
\|\nabla_{\xi}R(\cdot, e_n\eta_n)\|_{L^{\infty}(Q)}\le \|R(\cdot,
e_n\eta_n)\|^{\frac{\alpha'}{1+\alpha'}}_{L^{\infty}(Q)}|\nabla_{\xi}R(\cdot,
e_n\eta_n)|_{\alpha', Q}^{\frac{1}{1+\alpha'}} \ ,
\end{eqnarray}

where $C$ depends on $M,\alpha, \bar{A},n$ and $\bar{\gamma}$
only. By the above estimate and \eqref{stimaR} we get

\begin{eqnarray}
\ \ \ \ \ \ \ \ |\nabla_{\xi}R(\xi,e_n\eta_n)|\le
\frac{C}{r_1^{\beta}}h^{\beta +1 -n}, \ \mbox{for every}\ \xi \in
B^+_{\frac{r_1}{4}} \ \mbox{and}\  \eta\in
\left(-\frac{r_1}{4},0\right),
\end{eqnarray}

where $C$ depends on $M,\alpha, \bar{A},n$ and $\bar{\gamma}$
only. The thesis follows with $\beta=
\frac{\alpha'^2}{1+\alpha'}$.
\end{proof}

\begin{proof}[Proof of Theorem \ref{teorema stime asintotiche}.]
We first assume that the auxiliary hypothesis that $A(0)=I$ is fulfilled and denote with $H(\xi,\eta)$ the half space fundamental
solution of the operator
$\mbox{div}_{\xi}((1+(k-1))\chi^+(\xi)I(\xi)\nabla_{\xi})$ which
has the following explicit form

\begin{equation}\label{fund}
H(\xi,\eta)= \left\{
\begin{array}
{lcl} \frac{1}{k}\Gamma(\xi,\eta) + \frac{k-1}{k(k+1)}\Gamma(\xi,\eta^*)\ , &&\mbox{if}\ \xi_n,\eta_n>0\\
\frac{2}{k+1}\Gamma(\xi,\eta) \ , &&\mbox{if}\ \xi_n\cdot \eta_n<0\ \\
\Gamma(\xi,\eta) + \frac{1-k}{k+1}\Gamma(\xi,\eta^*)\ ,
&&\mbox{if}\ \xi_n,\eta_n<0
\end{array}
\right.
\end{equation}

where $\Gamma$ is the distribution introduced in \eqref{GC} and for any $\xi=(\xi',\xi_n)$ we denote $\xi^*=(\xi',-\xi_n)$. Let $\eta_n\in (-\frac{r_1}{4},0)$, then we have that

\begin{equation*}
\left\{
\begin{array}
{lcl} \mbox{div}_{\xi}((1+(k-1))\chi^+(\xi)I(\xi)\nabla_{\xi}(G_0(\xi,e_n\eta_n) - H(\xi, e_n\eta_n)))= 0 \ , \ \mbox{in} \ Q_{\frac{r_1}{2}}\ , \\
|(G_0(\xi,e_n\eta_n) - H(\xi, e_n\eta_n))|\le C r_1^{n-2} ,\ \ \ \
\mbox{for any}\ \xi\in \partial  Q_{\frac{r_1}{2}}\ .
\end{array}
\right.
\end{equation*}

Hence by the maximum principle we can infer that

\begin{eqnarray}\label{nosingsol}
\|G_0(\cdot,e_n\eta_n)-H(\cdot,
e_n\eta_n)\|_{L^{\infty}(Q_{\frac{r_1}{2}})}\le  C r_1^{n-2}
\end{eqnarray}

and by Theorem \ref{rego} we deduce that

\begin{eqnarray}\label{nosinggra}
\|\nabla_{\xi}G_0(\cdot,e_n\eta_n)-\nabla_{\xi}H(\cdot,
e_n\eta_n)\|_{L^{\infty}(Q_{\frac{r_1}{4}})}\le  C r_1^{n-1} \ .
\end{eqnarray}

We now consider a point $x\in \Phi^{-1}(B_{\frac{r_1}{4}}^{+})$
and $y_n\in (-\frac{r_1}{2},0)$, then we observe that being
$\Phi(y)=y$ we have that

\begin{eqnarray}
|\Phi(y)|=|\Phi(y)-\Phi(0)|\le |\Phi(y)-\Phi(x)|\ .
\end{eqnarray}

Moreover, by \eqref{bili} and the above estimate we have that

\begin{eqnarray}
C^{-1}|x|\le |\Phi(x)|\le|\Phi(x)-\Phi(y)|+|\Phi(y)|\le C|x-y|\ .
\end{eqnarray}

By combining the above estimate with \eqref{stimah}, we infer
that

\begin{eqnarray}\label{stimahh}
|\Phi(x)-x|\le \frac{C}{r_0^{\alpha}}|x|^{1+\alpha}\le
\frac{C}{r_0^{\alpha}}|x-e_n y_n|^{1+\alpha} \ ,
\end{eqnarray}

where $C$ depends on $M$ and $\alpha$ only. Let $\{A_k\}_{k\geq 1}$ be a regularizing sequence for $A$ obtained by
convolution with a sequence of mollifiers, then we have that

\begin{eqnarray}
\|A_k\|_{C^{1}(\Omega)}\le 2\bar{A},\quad\mbox{for any}\ k\in
\mathbb{N}
\end{eqnarray}

and $A_k$ satisfies \eqref{unifellip}, with $A=A_k$, $k\in
\mathbb{N}$. Let us introduce the following function

\begin{eqnarray}
F_k : &&B_{r_0}\setminus \{{e_ny_n}\} \rightarrow \mathbb{R}\\
    && z \mapsto  <A_k(z)(z-e_ny_n),(z-e_ny_n)>^{\frac{2-n}{2}},
\end{eqnarray}

where $<\cdot,\cdot>$ denotes the Euclidean scalar product of vectors in $\mathbb{R}^n$. Given $z_1,z_2\in B_{r_0}\setminus \{{e_ny_n}\}$ by the Mean-Value
Theorem, there exists $t_k, 0<t_k<1$ such that

\begin{eqnarray*}
|F_k(z_1)-F_k(z_2)| & &\le C|z_1-z_2| \Big(|<A_k(z_{t_k})(z_{t_k}-e_ny_n),(z_{t_k}-e_ny_n)>^{\frac{1-n}{2}}|\\
 &&+ \Big|<A_k(z_{t_k})(z_{t_k}-e_ny_n),(z_{t_k}-e_ny_n)>^{-\frac{n}{2}}\Big|\\
& & \times\Big| <\sum_{i=1}^n\partial_{z_i}Ak(z_{t_k})(z_{t_k}-e_ny_n),(z_{t_k}-e_ny_n)> \Big|\Big)
\end{eqnarray*}

where $z_{t_k}=z_1 +{t_k}(z_2-z_1)$ and where $C$ depends on
depends on $M,\alpha, \bar{A}$ and $n$ only. Let us denote with $\Gamma_k$ the fundamental solution introduced
in \eqref{GC2} associated to the matrix $A_k$. We choose
$z_1=\Phi(x)$ and $z_2=x$ and we have that


\begin{eqnarray*}
|\Gamma_k(\Phi(x), e_n y_n)-\Gamma_k(x, e_n y_n)| \le C|\Phi(x)-x||x-e_n y_n +t_k(\Phi(x)-x)|^{1-n}\ ,
\end{eqnarray*}

$C$ depends on depends on $M,\alpha, \bar{A}, \lambda$ and $n$
only. By \eqref{stimahh} and the triangle inequality  we deduce that for
any $x\in D_{j_{l+1}}\cap B_{\frac{r_0}{{4C}^{1/\alpha}}}$ we get

\begin{eqnarray}
&&|x-e_ny_n-t_k(\Phi(x)-x)|\ge |x-e_ny_n|-|t_k||\Phi(x)-x| \\
&&\ge |x-e_ny_n|-|x-e_ny_n|^{1+\alpha}\ge \frac{1}{2}|x-e_n-y_n| \
.
\end{eqnarray}

Finally combining the above estimates and \eqref{stimahh} we
obtain

\begin{eqnarray}\label{stimavicinanza2}
|\Gamma_k(\Phi(x), e_n y_n)-\Gamma_k(x, e_n y_n)| \le C
|x-e_ny_n|^{2-n +\alpha},
\end{eqnarray}

where $C$ depends on $M,\alpha, \lambda, \bar{A}$ and $n$ only.
Now since $A_k$ converges uniformly to $A$ in $\overline{\Omega}$ we can infer that

\begin{eqnarray}\label{stimavicinanza}
|\Gamma(\Phi(x), e_n y_n)-\Gamma(x, e_n y_n)| \le C,
|x-e_ny_n|^{2-n +\alpha},
\end{eqnarray}

for $x\in  \Phi^{-1}(B_{\frac{r_1}{4}}^{+})$, where $C$ depends on $M,\alpha, \lambda, \bar{A}$ and $n$ only. By \eqref{nosingsol}, \eqref{nosinggra} and \eqref{stimavicinanza}
we have

\begin{eqnarray}\label{soluzione}
|G_0(\Phi(x), e_ny_n) - H(x, e_n y_n)| &\le & |G_0(\Phi(x), e_ny_n) - H(\Phi(x), e_n y_n)|\nonumber\\
& + & |H(\Phi(x), e_ny_n) - H(x, e_n y_n)|\nonumber \\
&\le &\frac{C}{r_0^{\alpha}}|x-e_n y_n|^{\alpha +2-n} \
\end{eqnarray}

and

\begin{eqnarray}\label{gradiente}
|\nabla G_0(\Phi(x), e_ny_n) - \nabla H(x, e_n y_n)| \le
\frac{C}{r_0^{\alpha}}|x-e_n y_n|^{\alpha +1-n} \ ,
\end{eqnarray}

for $x\in  \Phi^{-1}(B_{\frac{r_1}{4}}^{+})$, where $C$
depends on $M,\lambda, \bar{\gamma}, \alpha$ and $n$ only. Moreover, by Lemma \ref{lemmaasi}, \eqref{bili} and recalling that
$\Phi(y)=y$, we get

\begin{eqnarray}\label{resto}
\ \ \ \ |R(\Phi(x),e_n\eta_n)| +
|x-e_n\eta_n||\nabla_{\xi}R(\Phi(x),\eta)|\le
\frac{c}{r_1^{\beta}}|\xi-e_n\eta_n|^{\beta+2-n} ,
\end{eqnarray}

for $x\in  \Phi^{-1}(B_{\frac{r_1}{4}}^{+})$ and where $C$
depends on $M,\lambda, \bar{\gamma}, \alpha$ and $n$ only. Gathering \eqref{soluzione}, \eqref{gradiente}, \eqref{resto} and
recalling that

\begin{eqnarray}
G(\bar{x},e_ny_n)= G_0(\Phi(\bar{x}),e_n y_n) +
R(\Phi(\bar{x}),e_n y_n)
\end{eqnarray}

we first find that

\begin{eqnarray}\label{asyfunsemplice}
\left|G(\bar{x}, e_n y_n) - \frac{1}{1+k}\Gamma(\bar{x},e_n
y_n)\right|\le\frac{C}{r_0^{\beta}}|\bar{x}-e_n y_n|^{\beta+2-n} \
\ ,
\end{eqnarray}

\begin{eqnarray}\label{asygrasemplice}
\ \  \ \ \ \left|\nabla_x G(\bar{x},e_n y_n ) -
\frac{1}{1+k}\nabla_x\Gamma(\bar{x},e_n
y_n)\right|\le\frac{C}{r_0^{\beta}}|\bar{x}-e-n y_n|^{\beta+1-n} \
\ ,
\end{eqnarray}

for a.e $\bar{x}\in D_{j_{l+1}}\cap
B_{\frac{r_0}{(4C)^{1/{\alpha}}}}$ and $y_n\in
(-r_1/(4C)^{1/(\alpha)},0)$, where $C$ depends on $M,\lambda, \bar{\gamma}$, $\bar{A}$, $\alpha$ and $n$ only. The
thesis then follows for the case $A(0)=I$.\\

To treat the general case when $A(0)\neq I$, we introduce the fundamental solution $H_{A(0)}$ of the operator
$\mbox{div}_{\xi}((1+(k-1))\chi^+(\xi)A(0)\nabla_{\xi})$. We set
$\sigma_I(\xi)= (1+(k-1))\chi^+(\xi)Id $ and $\sigma_{A(0)}(\xi)= (1+(k-1))\chi^+(\xi)A(0)$. Let us introduce the linear change of variable

\begin{eqnarray}
L: &&\mathbb{R}^n \rightarrow \mathbb{R}^n\\
   && \xi \mapsto x=L\xi := R\sqrt{A^{-1}(0)}\xi,
\end{eqnarray}

where $R$ is the planar rotation in $\mathbb{R}^n$ that rotates the unit vector $\frac{v}{||v||}$, where $v=\sqrt{A(0)}e_n$, to the n-th standard unit vector $e_n$ and such that

\[R|_{(\pi)^{\bot}}\equiv Id|_{(\pi)^{\bot}},\]

where $\pi$ is the plane in $\mathbb{R}^n$ generated by $e_n$, $v$ and $(\pi)^{\bot}$ denotes the orthogonal complement of $\pi$ in $\mathbb{R}^{n}$. For this choice of $L$ we have

\begin{description}
\item [i)] $A(0)=L^{-1}\cdot (L^{-1})^T$\ ,
\item [ii)] $(L\xi)\cdot e_n=\frac{1}{||v||} \xi\cdot e_n$.
\end{description}

which leads to

\[\sigma_{A(0)}(\xi)=L^{-1}\sigma_I(L \xi)(L^{-1})^T,\]

which means that $L^{-1}:x\mapsto\xi$ is the linear change of variables that maps $\sigma_I(x)$ into $\sigma_{A(0)}(\xi)$. Therefore the fundamental solution for the operator
$\mbox{div}_{\xi}((1+(k-1))\chi^+(\xi)A(0)\nabla_{\xi})$ turns out to be

\begin{equation}\label{funda0}
H_{A(0)}(\xi,\eta)= \left\{
\begin{array}
{lcl}\sqrt{\mbox{det} A^{-1}(0)}\left( \frac{1}{k}\Gamma(L\xi,L \eta) + \frac{k-1}{k(k+1)}\Gamma(L\xi,L^*\eta)\right)\ , &&\mbox{if}\ \xi_n,\eta_n>0\\
\sqrt{\mbox{det} A^{-1}(0)}\left(\frac{2}{k+1}\Gamma(L\xi,L\eta) \right)\ , &&\mbox{if}\ \xi_n\cdot \eta_n<0\ \\
\sqrt{\mbox{det} A^{-1}(0)}\left(\Gamma(L\xi,L \eta) + \frac{1-k}{k+1}\Gamma(L\xi,L^* \eta)\right)\ ,
&&\mbox{if}\ \xi_n,\eta_n<0
\end{array}
\right.
\end{equation}

where the matrix $L^*=\{l^*_{i,j}\}_{i,j=1}^n$ is such that $l_{i,j}^*=l_{i,j}$ for $i=1,\dots,n-1, j=1, \dots, n$ and $l_{n,j}^*=-l_{n,j}$ for $j=1\dots, n$. In particular we have that when $\xi_n\cdot \eta_n<0$ then

\[H_{A(0)}(\xi,\eta)=\sqrt{\mbox{det} A^{-1}(0)}\frac{2}{k+1} <A^{-1}(0)(\xi-\eta, \xi-\eta)>^{\frac{2-n}{2}}.\]

Hence for the case $A(0)\neq I$
\eqref{asyfunsemplice} and \eqref{asygrasemplice}
shall be replaced by

\begin{eqnarray*}
\left|G(\bar{x}, e_n y_n) - \frac{1}{1+k}<A^{-1}(0)(\bar{x} -e_ny_n),(\bar{x} -e_ny_n )>^{\frac{2-n}{2}} \right|\le\frac{C}{r_0^{\beta}}|\bar{x}-e_n y_n|^{\beta+2-n}
\ \ ,
\end{eqnarray*}

\begin{eqnarray*}
\left|\nabla_x G(\bar{x},e_n y_n ) -
\frac{1}{1+k}\nabla_x <A^{-1}(0)(\bar{x} -e_ny_n),(\bar{x} -e_ny_n )>^{\frac{2-n}{2}}  \right|\le\frac{C}{r_0^{\beta}}|\bar{x}-e_n y_n|^{\beta+1-n}
\ \ ,
\end{eqnarray*}

for  $\bar{x}\in D_{j_{l+1}}\cap
B_{\frac{r_0}{(4C)^{1/{\alpha}}}}$ and $y_n\in
(-r_1/(4C)^{1/(\alpha)},0)$ where $C$ depends on $M,\lambda, \bar{\gamma},$ $\bar{A}$ $\alpha$ and $n$ only.
Hence the thesis follows also for the general case.
\end{proof}


\subsection{\normalsize{Proof of unique continuation estimates}}\label{UC}
\setcounter{equation}{0}

Let $P_1$, $D_0$ $\Omega_0$, $(D_0)_r$ and $\tilde{G}_i$, for
$i=1,2$ be as in subsection \ref{subsection Green function}. Let
us fix $k\in\{2,\dots N\}$ and recall that there exist $j_1,\dots
j_K \in\{2,\dots N\}$ such that

\[D_{j_1}=D_1,\dots D_{j_K}=D_k.\]

We recall that

\[\mathcal{W}_K = \bigcup_{i=0}^{K}D_{j_i},\qquad \mathcal{U}_k = \Omega_0\setminus\overline{\mathcal{W}_K},\quad\textnormal{when}\:k\geq 0\]

($D_{j_0}=D_0$) and for any $y,z\in\mathcal{W}_K$

\[\tilde{S}_{\mathcal{U}_K}(y,z)=\int_{\mathcal{U}_K}(\tilde\sigma_A^{(1)}-\tilde\sigma_A^{(2)})\nabla\tilde{ G}_1(\cdot,y)\cdot\nabla\tilde{G}_2(\cdot,z),\quad \textnormal{when}\:k\geq 0.\]


The proof of Proposition 3.5 is a straight forward consequence of
the follwoing result (see \cite{A-V}[proof of Proposition 4.6]).

\begin{proposition}\label{proposizione pre unique continuation}
Let $v$ be a weak solution to

\[\mbox{div}\left(\tilde\sigma\nabla v\right)=0,\quad\mbox{in}\:\mathcal{W}_k,\]

where $\tilde\sigma$ is either equal to $\tilde\sigma_A^{(1)}$ or to
$\tilde\sigma_A^{(2)}$. Assume that, for given positive numbers
$\varepsilon_0$ and $E_0$, v satisfies

\begin{equation}\label{v condition 1}
|v(x)|\leq \varepsilon_0
r_0^{2-n},\quad\mbox{for\:every}\:x\in(D_0)_{\frac{r_0}{3}},
\end{equation}

and

\begin{equation}\label{vcondition 2}
|v(x)|\leq E_0 \left(r_0
d(x)\right)^{1-n/2},\quad\mbox{for\:every}\:x\in\mathcal{W}_k,
\end{equation}

where $d(x)=dist(x,\Sigma_{k+1})$. Then the following inequality
holds true for every $r\in(0,d_1]$

\begin{equation}
\left|v\left(w_{\bar{h}}(P_{k+1})\right)\right| \leq
r_0^{2-n}C^{\bar{h}}(E_0+\varepsilon_0)\left(\omega_{1/C}^{(k)}\left(\frac{\varepsilon_0}{E_0+\varepsilon_0}\right)\right)
^{\left(1/C\right)^{\bar{h}}}.
\end{equation}

\end{proposition}

\begin{proof} We observe that the proof of this result follows the same line of the argument used in \cite{A-V}[proof of Proposition 4.4] which is independent from the presence of isotropy/anisotropy in $\tilde\sigma$. In fact their proof is based on an argument of unique continuation which require $\tilde\sigma$ to be Lipschitz continuous and the interfaces between each domain $D_j$ to contain a $C^{1,\alpha}$ portion, therefore we simply recall \cite{A-V}[proof of Proposition 4.4] for a complete proof of this proposition. Here we simply recall for sake of completeness the main fact proven in \cite{A-V}[proof of Proposition 4.4].
By defining the quantities

\[r_1=\frac{r_0}{4},\qquad \bar\rho=\frac{r_1}{128\sqrt{1+L^2}}\]

let $y_m\in D_m$ be a point "near the portion" $\Sigma_{m+1}$ of the
interface between $D_m$ and $D_{m+1}$ defined by

\[y_m=P_{m+1}-\frac{r_1}{32}\nu(P_{m+1}),\]

where $P_{m+1}\in\Sigma_{m+1}$. Their main point is the proof of the following fact




\begin{equation}\label{claim 4.5 AV}
||v||_{L^{\infty}(B_{\bar\rho}(y_m))}\leq
r_0^{2-n}C^{m+1}(E_0+\varepsilon_0)\omega_{\frac{1}{C}}^{(m+1)}\left(\frac{\varepsilon_0}{E_0+\varepsilon_0}\right),
\end{equation}

where $\bar\rho$ has been chosen above so that
$B_{\bar\rho}(y_m)\subset D_m$. The proof of the above inequality is done by induction. A so-called argument of \textit{global propagation of smallness} is used there to prove \eqref{claim 4.5 AV} for $m=0$. We refer to \cite{A-R-R-V}, Theorem 5.3 for a complete treatment of this
topic. The rest of the proof is based on the \textit{three sphere inequality}, therefore we simply refer to \cite{A-V}[proof of Proposition 4.4] for this.

\end{proof}

\section*{\normalsize{Acknowledgments}}
The authors gratefully acknowledge the fruitful
conversations with G. Alessandrini who kindly exchanged ideas
about the global stability issue with the authors during the
preparation of this work.

\end{document}